\documentclass[11pt]{article}

\usepackage[T1]{fontenc}
\usepackage{lmodern}
\usepackage{amsmath,amssymb,amsthm,mathtools}
\usepackage[hidelinks]{hyperref}

\newtheoremstyle{harduxdefinition}
{8pt plus 2pt minus 1pt}{8pt plus 2pt minus 1pt}
{\normalfont}{}{\bfseries}{.}{0.5em}{}
\newtheoremstyle{harduxplain}
{8pt plus 2pt minus 1pt}{8pt plus 2pt minus 1pt}
{\itshape}{}{\bfseries}{.}{0.5em}{}
\newtheoremstyle{harduxremark}
{8pt plus 2pt minus 1pt}{8pt plus 2pt minus 1pt}
{\normalfont}{}{\itshape}{.}{0.5em}{}
\newtheorem{definition}{Definition}[section]
\usepackage{amsthm}
\newtheorem{lemma}{Lemma}[section]
\newtheorem{proposition}{Proposition}[section]
\newtheorem{theorem}{Theorem}[section]
\newtheorem{corollary}{Corollary}[section]
\theoremstyle{harduxremark}
\newtheorem{remark}{Remark}[section]
\newtheorem{example}{Example}[section]

\makeatletter
\def\bstctlcite{\@ifnextchar[{\@bstctlcite}{\@bstctlcite[@auxout]}}
\def\@bstctlcite[#1]#2{\@bsphack
  \@for\@citeb:=#2\do{%
    \edef\@citeb{\expandafter\@firstofone\@citeb}%
    \if@filesw
      \immediate\write\csname #1\endcsname{\string\citation{\@citeb}}%
    \fi}%
  \@esphack}
\makeatother

\DeclareMathOperator{\gra}{gra}

\title{Sharp Lower Bounds on the Haraux Function Beyond Reflexivity}
\author{Weifeng Yang\\
\texttt{ywf841673182@gmail.com}
}
\date{}

\begin{document}
\bstctlcite{IEEEexample:BSTcontrol}

\maketitle

\begin{abstract}
We prove that the sharp $\frac{1}{2}$ lower bound for the Haraux function holds for every maximally monotone operator of type~(NI) on an arbitrary real Banach space. This extends the result established in reflexive Banach spaces to arbitrary real Banach spaces. We also establish an exact decomposition at each graph point, showing that the local contribution to the Haraux function and a nonnegative residual sum to one half of the weighted squared displacement.  
Due to the equivalence between type~(NI) and quasidensity, this decomposition also yields the sharp bound without requiring a graph point at which the residual vanishes. 
Moreover, for every operator with a nonempty graph, this decomposition yields a lower bound involving the residual infimum, and for maximally monotone operators it further yields a new characterization of type~(NI) in terms of the Haraux function. 
Finally, on $c_0$, we give a maximally monotone operator of type~(NI) for which the residual infimum is zero at some target but is not attained. This shows that the existence of a graph point at which the residual vanishes is strictly stronger than the vanishing of the residual infimum required in our proof.

\end{abstract}



\section{Introduction}
\label{sec:introduction}

Let $X$ be a real Banach space and let $A:X\rightrightarrows X^*$ be an operator. The operator is monotone when $\langle x-y,x^*-y^*\rangle\geq0$ for every $(x,x^*),(y,y^*)\in\gra A$, and it is maximally monotone when no proper extension of its graph preserves this property. For a maximally monotone operator, membership in the graph can be tested by the Haraux function as follows. 
\begin{equation}
H_A(x,u^*)
=
\sup_{(y,y^*)\in\gra A}
\langle x-y,y^*-u^*\rangle.
\label{eq:haraux-intro}
\end{equation}
This function arose in monotone range theory and is central to the Br\'ezis-Haraux theorem \cite{brezisharaux1976}. Its role as a graph test follows from its relation to the Fitzpatrick function. In fact, if $F_A$ denotes the Fitzpatrick function \cite{fitzpatrick1988}, then
\begin{equation}
F_A(x,u^*)=H_A(x,u^*)+\langle x,u^*\rangle.
\label{eq:haraux-fitzpatrick}
\end{equation}
Since the Fitzpatrick representation gives $H_A\geq0$ with equality exactly on $\gra A$, $H_A$ provides a qualitative test for graph membership, and this test is used in many applications, such as nonautonomous monotone evolution equations \cite{attouchcabotczarnecki2018,ibaraki2024} and composite monotone inclusions \cite{combettesmayrand2026,penotratsimahalo2001}. Through Eq. (\ref{eq:haraux-fitzpatrick}), a graph distance lower bound for $H_A$ is also a quantitative estimate for the Fitzpatrick gap. This places the present problem within the development of quantitative Fenchel-Young and Fitzpatrick-type inequalities. Such inequalities have recently appeared in inverse problems over measures \cite{andradepeyrepoon2026} and learning losses \cite{rakotomandimby2024}. Additionally, Bauschke, McLaren, and Sendov refined the Fenchel-Young inequality through Fitzpatrick functions for subdifferentials \cite{bauschkemclarensendov2006}. Voisei and Z\u{a}linescu established Banach space estimates for strongly representable operators \cite{voiseizalinescu2009}. In Hilbert spaces, Carlier proved a quantitative Fenchel-Young inequality with applications to Fitzpatrick functions \cite{carlier2023}, and Bauschke, Singh, and Wang subsequently refined it \cite{bauschkesinghwang2023}. Burachik and Mart\'inez-Legaz extended this line of results to reflexive Banach spaces \cite{burachikmartinezlegaz2025}.  

Recently, Combettes and Mayrand \cite{combettesmayrand2026} proved that the Haraux function of a maximally monotone operator satisfies the sharp estimate on the reflexive real Banach spaces as follows.
\begin{equation}
H_A(x,u^*)
\geq
\frac{1}{2} d_{\gra A,\gamma}^2(x,u^*),
\qquad \gamma>0,
\label{eq:sharp-intro}
\end{equation}
where $d_{\gra A,\gamma}$ is the distance generated by the weighted product norm. On the other hand, in an arbitrary real Banach space, the result of Voisei and Z\u{a}linescu \cite{voiseizalinescu2009} gives the weaker estimate for a maximally monotone operator of type~(NI) 
\begin{equation}
H_A(x,u^*)
\geq
\frac{1}{4} d_{\gra A,1}^2(x,u^*).
\label{eq:quarter-intro}
\end{equation}

However, these two estimates do not reveal whether the gap constant from $\frac{1}{2}$ to $\frac{1}{4}$ is caused by nonreflexive geometry or by the proof mechanism used to obtain the sharper bound. Moreover, the proof of Eq. (\ref{eq:sharp-intro}) relies on a graph point satisfying an exact metric resolvent condition, but reflexivity only guarantees the existence of this point, whereas the same condition may have no graph solution in a nonreflexive space \cite{combettesmayrand2026}.

More importantly, the failure of exact solvability does not imply that the sharp bound fails. To this end, we introduce the following residual and establish a new exact decomposition relating it to the Haraux function. For $\gamma>0$, we set  
\begin{equation}
\mathcal R_\gamma(a,a^*)
=
\frac{\|a\|^2}{2\gamma}
+\frac{\gamma}{2}\|a^*\|_*^2
+\langle a,a^*\rangle.
\label{eq:residual-intro}
\end{equation}
For every $(y,y^*)\in\gra A$, the residual satisfies
\begin{equation}
\begin{aligned}
&\langle x-y,y^*-u^*\rangle
+\mathcal R_\gamma(y-x,y^*-u^*)\\
&\qquad=
\frac{1}{2}\left(
\frac{\|x-y\|^2}{\gamma}
+\gamma\|u^*-y^*\|_*^2
\right).
\end{aligned}
\label{eq:identity-intro}
\end{equation}
Moreover, for a maximally monotone operator of type~(NI), quasidensity guarantees graph points with arbitrarily small residual \cite{simons2020}. Therefore, by combining Eq. (\ref{eq:identity-intro}) with this approximation property, we can obtain the sharp $\frac{1}{2}$ bound. 

Beyond this sharp extension, Eq. (\ref{eq:identity-intro}) yields a lower bound involving the residual infimum for every operator with a nonempty graph. For maximally monotone operators, the residual infimum is zero at every target for every $\gamma>0$ exactly when the operator is of type~(NI).   
Thus, Eq. (\ref{eq:identity-intro}) yields a new characterization of type~(NI). We also provide an example (Example~\ref{ex:nonattainment}) showing that the residual infimum may be zero even though no graph point makes the residual vanish.  These results lead to the following three contributions.

\subsection{Contributions}

The contributions of this paper are as follows.

(1) We prove that the sharp $\frac{1}{2}$ lower bound in terms of $d_{\gra A,\gamma}$ holds beyond reflexivity. Specifically, the bound holds for every maximally monotone operator of type~(NI) on an arbitrary real Banach space and for every weight $\gamma>0$, and $\frac{1}{2}$ remains the best uniform constant.

(2) We establish an exact pointwise decomposition that connects the local contributions to the Haraux function with weighted squared displacement. 
At every graph point, the local contribution to the Haraux function and the residual $\mathcal R_\gamma$ sum to $\frac{1}{2}$ times the weighted squared displacement. 
This pointwise decomposition also yields a global lower bound for every operator with a nonempty graph. 
Moreover, for maximally monotone operators, this pointwise decomposition yields a new characterization of type~(NI) in terms of the Haraux function.

(3) We remove exact metric resolvent solvability from the proof of the sharp bound. The reflexive argument uses a graph point where $\mathcal R_\gamma=0$, but the pointwise decomposition shows that graph points with arbitrarily small residual are sufficient. Through quasidensity, type~(NI) provides these graph points on an arbitrary real Banach space and yields the sharp $\frac{1}{2}$ bound without reflexivity. 
We also provide an example on $c_0$ (Example~\ref{ex:nonattainment}) in which graph points with arbitrarily small residual exist even though no graph point makes the residual vanish.

The paper is organized as follows. Section~\ref{sec:preliminaries} introduces some preliminary notation and background. Section~\ref{sec:bound} proves the weighted form of quasidensity, establishes the exact decomposition, derives its consequences for the residual infimum, and proves the sharp lower bound. Section~\ref{sec:certificates} investigates exact and approximate residual vanishing, distinguishes the vanishing of the residual infimum from its attainment, and derives a graph distance lower bound for the Fenchel-Young gap.  Conclusions are presented in Section \ref{sec:conclusion}.

\section{Preliminaries}
\label{sec:preliminaries}

Throughout this paper, $X$ is an arbitrary real Banach space, $X^*$ is its topological dual, and $\langle\cdot,\cdot\rangle$ denotes the canonical duality pairing. The norms of $X$ and $X^*$ are denoted by $\|\cdot\|$ and $\|\cdot\|_*$, respectively. The canonical image of $x\in X$ in $X^{**}$ is denoted by $\widehat{x}$. 

\begin{definition}[Monotonicity and type~(NI)]
\label{def:operator-class}
Let $A:X\rightrightarrows X^*$. Its graph and domain are
\begin{equation}
\gra A=\{(x,x^*)\in X\times X^*\mid x^*\in Ax\}
\quad\text{and}\quad
\operatorname{dom}A=\{x\in X\mid Ax\neq\emptyset\}.
\label{eq:graph-domain}
\end{equation}
The operator $A$ is monotone if
\begin{equation}
\langle x-y,x^*-y^*\rangle\geq0,
\label{eq:monotonicity}
\end{equation}
for all $(x,x^*),(y,y^*)\in\gra A$. It is maximally monotone if its graph is maximal, with respect to set inclusion, among monotone subsets of $X\times X^*$.

A maximally monotone operator $A$ is of type~\emph{(NI)} if, for every $(u^*,x^{**})\in X^*\times X^{**}$,
\begin{equation}
\inf_{(y,y^*)\in\gra A}
\langle y^*-u^*,\widehat y-x^{**}\rangle
\leq 0.
\label{eq:type-ni}
\end{equation}
\end{definition}

For a maximally monotone operator $A$, a point $(x,x^*)\in X\times X^*$ belongs to $\gra A$ if and only if 
\begin{equation}
\langle x-y,x^*-y^*\rangle\geq0
\quad \forall (y,y^*)\in\gra A.
\label{eq:maximality-test}
\end{equation}
Through the Fitzpatrick representation theorem, this characterization also gives the nonnegativity of $H_A$ and identifies its zero set with $\gra A$. 

To compare the Haraux function with the distance to the graph, we first introduce the quantities used in the comparison. 

\begin{definition}[Local graph quantities]
\label{def:local-energies}
Let $A:X\rightrightarrows X^*$ be an operator. Let $(x,u^*)\in X\times X^*$ and $\gamma>0$. For a graph point $(y,y^*)\in\gra A$, define its Haraux contribution by
\begin{equation}
h_A(x,u^*;y,y^*)
=\langle x-y,y^*-u^*\rangle.
\label{eq:local-haraux}
\end{equation}
Thus,
\begin{equation}
H_A(x,u^*)
=\sup_{(y,y^*)\in\gra A}h_A(x,u^*;y,y^*).
\label{eq:haraux-formal}
\end{equation}
The weighted squared displacement from $(x,u^*)$ to $(y,y^*)$ is
\begin{equation}
D_\gamma^2((x,u^*),(y,y^*))
=\frac{\|x-y\|^2}{\gamma}
+\gamma\|u^*-y^*\|_*^2,
\label{eq:weighted-displacement}
\end{equation}
and the corresponding squared graph distance is
\begin{equation}
d_{\gra A,\gamma}^2(x,u^*)
=\inf_{(y,y^*)\in\gra A}
D_\gamma^2((x,u^*),(y,y^*)).
\label{eq:weighted-distance}
\end{equation}
Finally, for $(a,a^*)\in X\times X^*$, define the weighted residual
\begin{equation}
\mathcal R_\gamma(a,a^*)
=\frac{\|a\|^2}{2\gamma}
+\frac{\gamma}{2}\|a^*\|_*^2
+\langle a,a^*\rangle.
\label{eq:residual-formal}
\end{equation}
\end{definition}

We use the following notation for the residual infimum along the graph. 

\begin{definition}[Residual infimum]
\label{def:weighted-defect}
Let $A:X\rightrightarrows X^*$ be an operator with nonempty graph. For every $(x,u^*)\in X\times X^*$ and every $\gamma>0$, define
\begin{equation}
\delta_{A,\gamma}(x,u^*)
=
\inf_{(y,y^*)\in\gra A}
\mathcal R_\gamma(y-x,y^*-u^*).
\label{eq:weighted-defect}
\end{equation}
\end{definition}

It is also the weighted form of the infimum appearing in the definition of quasidensity.

We also introduce the definition of the quasidensity as follows.  
\begin{definition}[Quasidensity]
\label{def:quasidensity}
A subset $M\subset X\times X^*$ is quasidense if, for every $(x,u^*)\in X\times X^*$,
\begin{equation}
\inf_{(y,y^*)\in M}
\left(
\frac{1}{2}\|y-x\|^2
+\frac{1}{2}\|y^*-u^*\|_*^2
+\langle y-x,y^*-u^*\rangle
\right)
=0.
\label{eq:quasidensity}
\end{equation}
\end{definition}

Quasidensity was developed as a direct Banach space approximation property \cite{simons2016,simons2018}. For maximally monotone operators, it is equivalent to type~(NI) \cite{simons2020,simons2025faces}. 
A general maximally monotone operator on a nonreflexive space need not be quasidense \cite{simons2016,simons2018,bauschkeborweinwangyao2012}. 

Type~(NI) also coincides with Gossez type~(D) \cite{marquesalvessvaiter2010}, whose bounded-net formulation expresses the approximation of the monotone bidual extension by original graph points. Representative functions and bounded-net closure operations on general Banach spaces continue this line of analysis \cite{eberhardwenczel2021,eberhardwenczel2025}, and approximate duality mappings describe nearby graph values for type~(D) operators \cite{nguyennguyenhuynh2024}. These formulations provide the broader nonreflexive approximation context. However, these results cannot directly yield the sharp Haraux lower bound beyond reflexivity.

\section{The exact decomposition and the sharp bound}
\label{sec:bound}

To this end, we first extend the vanishing of the residual infimum from $\gamma=1$ to every $\gamma>0$ as follows, and then establish the exact pointwise decomposition that yields the sharp bound.

\begin{lemma}[Weighted quasidensity]
\label{lem:weighted-quasidensity}
Let $A:X\rightrightarrows X^*$ be maximally monotone of type~(NI), and let $\gamma>0$. Then, $\forall (a,a^*)\in X\times X^*$, we have $\mathcal R_\gamma(a,a^*)\geq0$. Moreover, $\forall (x,u^*)\in X\times X^*$, we have
\begin{equation}
\inf_{(y,y^*)\in\gra A}
\mathcal R_\gamma(y-x,y^*-u^*)
=0.
\label{eq:weighted-quasidensity}
\end{equation}
\end{lemma}

\begin{proof}
Fix $(a,a^*)\in X\times X^*$, from the duality inequality, the scalar Young inequality and Eq. (\ref{eq:residual-formal}), we have 
\begin{align}
\mathcal R_\gamma(a,a^*)
&\geq
\frac{\|a\|^2}{2\gamma}
+\frac{\gamma}{2}\|a^*\|_*^2
-\|a\|\|a^*\|_* \notag\\
&=
\frac{1}{2}\left(
\frac{\|a\|}{\sqrt\gamma}
-\sqrt\gamma\|a^*\|_*
\right)^2
\geq0.
\label{eq:residual-nonnegative}
\end{align}

Next, we fix $(x,u^*)\in X\times X^*$ and define $B:X\rightrightarrows X^*$ as follows. 
\begin{equation*}
Bz=\{\gamma z^*\mid z^*\in Az\}.
\end{equation*}
Then, if $(y,\gamma y^*)$ and $(z,\gamma z^*)$ belong to $\gra B$, we infer
\begin{align*}
\langle y-z,\gamma y^*-\gamma z^*\rangle
&=\gamma\langle y-z,y^*-z^*\rangle\\
&\geq0.
\end{align*}
Thus, $B$ is monotone. Let $(w,w^*)$ be monotonically related to $\gra B$, and $\forall (y,y^*)\in\gra A$, then we give
\begin{align*}
\left\langle w-y,\frac{w^*}{\gamma}-y^*\right\rangle
&=\frac{1}{\gamma}\langle w-y,w^*-\gamma y^*\rangle\\
&\geq0.
\end{align*}
The maximality of $A$ gives $\left(w,\frac{w^*}{\gamma}\right)\in\gra A$, from the above formulation, this implies $(w,w^*)\in\gra B$, thus $B$ is maximally monotone. 

Furthermore, for $(p^*,p^{**})\in X^*\times X^{**}$, the type~(NI) property of $A$ yields
\begin{align*}
&&\inf_{(y,\gamma y^*)\in\gra B}
\langle \gamma y^*-p^*,\widehat y-p^{**}\rangle&=
\gamma
\inf_{(y,y^*)\in\gra A}
\left\langle y^*-\frac{p^*}{\gamma},\widehat y-p^{**}\right\rangle \notag \\
&& &\leq0.
\end{align*}
This means $B$ is of type~(NI), and the equivalence between type~(NI) and quasidensity \cite{simons2020} shows that $\gra B$ is quasidense. Therefore, applying Definition~\ref{def:quasidensity} at $(x,\gamma u^*)$, we infer 
\begin{align*}
0
&=
\inf_{(y,\gamma y^*)\in\gra B}
\left(
\frac{1}{2}\|y-x\|^2
+\frac{1}{2}\|\gamma(y^*-u^*)\|_*^2
+\langle y-x,\gamma(y^*-u^*)\rangle
\right)\\
&=
\gamma
\inf_{(y,y^*)\in\gra A}
\mathcal R_\gamma(y-x,y^*-u^*).
\end{align*}
Since $\gamma>0$, this implies Eq. (\ref{eq:weighted-quasidensity}) is true.
\end{proof}

Next, we establish the exact pointwise decomposition in which the local contribution to the Haraux function and the residual  $\mathcal R_\gamma$ sum to one half of the weighted squared displacement.

\begin{proposition}[Exact decomposition]
\label{prop:decomposition}
Let $A:X\rightrightarrows X^*$ be any operator, $(x,u^*)\in X\times X^*$ and $\gamma>0$. Then, $\forall (y,y^*)\in\gra A$,
\begin{equation}
h_A(x,u^*;y,y^*)
+\mathcal R_\gamma(y-x,y^*-u^*)
=
\frac{1}{2}D_\gamma^2((x,u^*),(y,y^*)).
\label{eq:decomposition}
\end{equation}
\end{proposition}

\begin{proof}
By Definition~\ref{def:local-energies}, we infer 
\begin{align*}
&h_A(x,u^*;y,y^*)+\mathcal R_\gamma(y-x,y^*-u^*)\\
&=
\langle x-y,y^*-u^*\rangle
+\frac{\|y-x\|^2}{2\gamma}
+\frac{\gamma}{2}\|y^*-u^*\|_*^2
+\langle y-x,y^*-u^*\rangle\\
&=
\frac{\|x-y\|^2}{2\gamma}
+\frac{\gamma}{2}\|u^*-y^*\|_*^2
=\frac{1}{2}D_\gamma^2((x,u^*),(y,y^*)).
\end{align*}
\end{proof}

\begin{remark}[The geometric identification]
\label{rem:geometric-identification}
Proposition~\ref{prop:decomposition} shows that, at each graph point, the Haraux contribution and the residual $\mathcal R_\gamma$ sum exactly to $\frac{1}{2}$ times the weighted squared displacement. This pointwise decomposition requires no monotonicity assumption and no reflexivity assumption, and the type~(NI) assumption is used only to obtain graph points with arbitrarily small residual. 

\end{remark}

Furthermore, using the residual infimum in this decomposition gives the following global consequences. 

\begin{corollary}[Consequences of the decomposition]
\label{cor:defect-form}
Let $A:X\rightrightarrows X^*$ be an operator with nonempty graph. Then, $\forall (x,u^*)\in X\times X^*$ and $\forall \gamma>0$, we have 
\begin{equation}
\sup_{(y,y^*)\in\gra A}
\left[
h_A(x,u^*;y,y^*)
-\frac{1}{2}D_\gamma^2((x,u^*),(y,y^*))
\right]
=-\delta_{A,\gamma}(x,u^*),
\label{eq:defect-identity}
\end{equation}
and
\begin{equation}
H_A(x,u^*)+\delta_{A,\gamma}(x,u^*)
\geq
\frac{1}{2}d_{\gra A,\gamma}^2(x,u^*).
\label{eq:defect-corrected-bound}
\end{equation}

Additionally, if $A$ is maximally monotone, the following statements are equivalent:
\begin{enumerate}
\item[(i)] $A$ is of type~\emph{(NI)}. 
\item[(ii)] $\delta_{A,\gamma}(x,u^*)=0$ for every $(x,u^*)\in X\times X^*$ and every $\gamma>0$. 
\item[(iii)] for every $(x,u^*)\in X\times X^*$ and every $\gamma>0$, we have 
\begin{equation}
\sup_{(y,y^*)\in\gra A}
\left[
h_A(x,u^*;y,y^*)
-\frac{1}{2}D_\gamma^2((x,u^*),(y,y^*))
\right]
=0
\label{eq:haraux-side-characterization}
\end{equation}
\end{enumerate}
\end{corollary}

\begin{proof}
Fix $(x,u^*)\in X\times X^*$, $\gamma>0$, $\forall (y,y^*)\in\gra A$, from Proposition~\ref{prop:decomposition}, we have 
\begin{equation*}
h_A(x,u^*;y,y^*)
-\frac{1}{2}D_\gamma^2((x,u^*),(y,y^*))
=-\mathcal R_\gamma(y-x,y^*-u^*).
\end{equation*}
Then, taking the supremum on the left and the residual infimum on the right,  Eq. (\ref{eq:defect-identity}) can be proved. 

To prove Eq. (\ref{eq:defect-corrected-bound}), let $\varepsilon>0$,  Definition~\ref{def:weighted-defect} gives a point $(y_\varepsilon,y_\varepsilon^*)\in\gra A$ such that
\begin{equation*}
\mathcal R_\gamma(y_\varepsilon-x,y_\varepsilon^*-u^*)
<\delta_{A,\gamma}(x,u^*)+\varepsilon.
\end{equation*}
Using this point in the Haraux supremum, and then applying Proposition~\ref{prop:decomposition} and Eq. (\ref{eq:weighted-distance}),  we have 
\begin{align*}
H_A(x,u^*)+\delta_{A,\gamma}(x,u^*)
&\geq h_A(x,u^*;y_\varepsilon,y_\varepsilon^*)
+\delta_{A,\gamma}(x,u^*)\\
&>\frac{1}{2}D_\gamma^2((x,u^*),(y_\varepsilon,y_\varepsilon^*))-\varepsilon\\
&\geq\frac{1}{2}d_{\gra A,\gamma}^2(x,u^*)-\varepsilon.
\end{align*}
Let $\varepsilon \to 0$, then Eq. (\ref{eq:defect-corrected-bound}) is true.

Suppose that $A$ is maximally monotone. If (i) holds, Lemma~\ref{lem:weighted-quasidensity} gives
\begin{equation*}
\delta_{A,\gamma}(x,u^*)
=
\inf_{(y,y^*)\in\gra A}
\mathcal R_\gamma(y-x,y^*-u^*)
=
0,
\qquad
\forall (x,u^*)\in X\times X^*,
\quad
\forall \gamma>0.
\end{equation*}
Thus (ii) holds.

Conversely, assume (ii) holds. Let $\gamma=1$, we have 
\begin{equation*}
\inf_{(y,y^*)\in\gra A}
\left(
\frac{1}{2}\|y-x\|^2
+
\frac{1}{2}\|y^*-u^*\|_*^2
+
\langle y-x,y^*-u^*\rangle
\right)
=
0 ~(\forall (x,u^*)\in X\times X^*). 
\end{equation*}
Thus, $\gra A$ is quasidense. Since $A$ is maximally monotone, the equivalence between quasidensity and type~(NI) \cite{simons2020} yields (i). 

Finally, from Eq. (\ref{eq:defect-identity}), we have 
\begin{equation*}
\sup_{(y,y^*)\in\gra A}
\left[
h_A(x,u^*;y,y^*)
-
\frac{1}{2}D_\gamma^2((x,u^*),(y,y^*))
\right]
=
-\delta_{A,\gamma}(x,u^*).
\end{equation*}
Therefore, the supremum in (iii) is zero if and only if
$\delta_{A,\gamma}(x,u^*)=0$. This implies the equivalence of (ii) and (iii).
\end{proof}

Eq. (\ref{eq:haraux-side-characterization}) gives a new characterization of type~(NI) in terms of the Haraux function. Therefore, combining Eq. (\ref{eq:defect-corrected-bound}) with Lemma~\ref{lem:weighted-quasidensity},  we obtain the following sharp lower bound.

\begin{theorem}[Sharp lower bound]
\label{thm:main}
Let $X$ be an arbitrary real Banach space, and let $A:X\rightrightarrows X^*$ be maximally monotone of type~(NI). Then, $\forall (x,u^*)\in X\times X^*$ and $\forall \gamma>0$, we have 
\begin{equation}
H_A(x,u^*)
\geq
\frac{1}{2} d_{\gra A,\gamma}^2(x,u^*).
\label{eq:main-bound}
\end{equation}
\end{theorem}

\begin{proof}
Fix $(x,u^*)\in X\times X^*$ and $\gamma>0$, from $\gra A\neq\emptyset$ and Eq. (\ref{eq:weighted-distance}), we have 
\begin{equation*}
d_{\gra A,\gamma}^2(x,u^*)<+\infty.
\end{equation*}
If $H_A(x,u^*)=+\infty$, Eq. (\ref{eq:main-bound}) holds immediately. 

If not, assume that $H_A(x,u^*)<+\infty$, from Lemma~\ref{lem:weighted-quasidensity}, we have
\begin{equation*}
\inf_{(y,y^*)\in\gra A}
\mathcal R_\gamma(y-x,y^*-u^*)
=
0.
\end{equation*}
Thus, $\forall\varepsilon>0$, there exists
$(y_\varepsilon,y_\varepsilon^*)\in\gra A$ such that
\begin{equation*}
\mathcal R_\gamma
(y_\varepsilon-x,y_\varepsilon^*-u^*)
<
\varepsilon.
\end{equation*}
By the definition of $H_A$ and Proposition~\ref{prop:decomposition}, we infer 
\begin{align*}
H_A(x,u^*)
&\geq
h_A(x,u^*;y_\varepsilon,y_\varepsilon^*)\\
&=
\frac{1}{2}
D_\gamma^2
\bigl((x,u^*),(y_\varepsilon,y_\varepsilon^*)\bigr)
-
\mathcal R_\gamma
(y_\varepsilon-x,y_\varepsilon^*-u^*)\\
&>
\frac{1}{2}
D_\gamma^2
\bigl((x,u^*),(y_\varepsilon,y_\varepsilon^*)\bigr)
-
\varepsilon\\
&\geq
\frac{1}{2}
d_{\gra A,\gamma}^2(x,u^*)
-
\varepsilon,
\end{align*}
where the last inequality follows from Eq. (\ref{eq:weighted-distance}).
Let $\varepsilon\to 0$, thus Eq. (\ref{eq:main-bound}) is true.
\end{proof}

From Theorem \ref{thm:main}, our proof shows that exact metric resolvent solvability is not necessary for the sharp lower bound. Moreover, compared to the previous proof on the reflexive Banach spaces, which uses reflexivity to obtain a graph point with zero residual, our proof only requires graph points with arbitrarily small residual for maximally monotone operators of type~(NI). This weaker condition is sufficient to obtain the sharp constant $\frac{1}{2}$ without requiring exact residual attainment.

We also prove that the constant $\frac{1}{2}$ in Theorem~\ref{thm:main} is optimal by constructing an example for which the lower bound is attained with a positive value.

\begin{example}[Sharpness for every weight]
\label{ex:sharpness}
Let $X$ be a nonzero real Hilbert space and identify $X^*$ with $X$. Fix $\gamma>0$ and define $A=\gamma^{-1}I$. For $x,u\in X$, we set 
\begin{equation}
d=x-\gamma u
\quad\text{and}\quad
m=\frac{x+\gamma u}{2}.
\label{eq:sharp-midpoint}
\end{equation}
Writing $y=m+v$ gives $x-y=\frac{d}{2}-v$ and $y-\gamma u=\frac{d}{2}+v$. Substitution gives
\begin{align}
H_A(x,u)
&=\frac{1}{\gamma}\sup_{y\in X}\langle x-y,y-\gamma u\rangle \notag\\
&=\frac{1}{\gamma}\sup_{v\in X}
\left(\frac{1}{4}\|d\|^2-\|v\|^2\right)
=\frac{1}{4\gamma}\|x-\gamma u\|^2.
\label{eq:sharp-haraux}
\end{align}
Similarly,
\begin{align}
d_{\gra A,\gamma}^2(x,u)
&=\frac{1}{\gamma}\inf_{y\in X}
\bigl(\|x-y\|^2+\|\gamma u-y\|^2\bigr) \notag\\
&=\frac{1}{\gamma}\inf_{v\in X}
\left(2\|v\|^2+\frac{1}{2}\|d\|^2\right)
=\frac{1}{2\gamma}\|x-\gamma u\|^2.
\label{eq:sharp-distance}
\end{align}
If $x\neq\gamma u$, then
\begin{equation}
H_A(x,u)=\frac{1}{2}d_{\gra A,\gamma}^2(x,u)>0.
\label{eq:sharp-equality}
\end{equation}
Thus, $\forall \gamma>0$, there exists a maximally monotone operator of type~(NI) and a target point at which equality in Eq.~(\ref{eq:main-bound}) holds. This means that no larger uniform constant can replace $\frac{1}{2}$ in Theorem~\ref{thm:main}. 
\end{example}

\section{Residual vanishing and exact attainment}
\label{sec:certificates}

In this section, we investigate the distinction between exact residual vanishing and the vanishing of the residual infimum,  explaining why the sharp bound does not require exact residual attainment.

\subsection{Characterization of residual vanishing}

To characterize the graph points at which the residual vanishes, we first introduce the definitions as follows.

\begin{definition}[Convex-analytic notation]
\label{def:convex-analysis}
Let $\Gamma_0(X)$ denote the set of proper lower semicontinuous convex functions from $X$ to $(-\infty,+\infty]$. For $\varphi\in\Gamma_0(X)$, its conjugate and subdifferential are
\begin{align}
\varphi^*(u^*)
&=\sup_{x\in X}\bigl(\langle x,u^*\rangle-\varphi(x)\bigr),
\label{eq:conjugate-definition}\\
\partial\varphi(x)
&=\left\{
x^*\in X^*\;\middle|\;
\varphi(z)\geq\varphi(x)+\langle z-x,x^*\rangle
\text{ for every }z\in X
\right\}.
\label{eq:subdifferential}
\end{align}
\end{definition}

\begin{definition}[Duality mapping]
\label{def:duality-mapping}
Let
\begin{equation}
q(a)=\frac{1}{2}\|a\|^2.
\end{equation}
The normalized duality mapping is
\begin{equation}
J=\partial q.
\label{eq:duality-mapping}
\end{equation}
\end{definition}

Next, we prove the following lemma.  

\begin{lemma}[Quadratic conjugacy]
\label{lem:quadratic-conjugacy}
Let $\gamma>0$ and define $q_\gamma:X\to\mathbb R$ by  
\begin{equation*}
q_\gamma(a)=\frac{\|a\|^2}{2\gamma},
\qquad
\forall a\in X.
\end{equation*}
Then, $\forall b^*\in X^*$, we have 
\begin{equation}
q_\gamma^*(b^*)=\frac{\gamma}{2}\|b^*\|_*^2.
\label{eq:quadratic-conjugate}
\end{equation}
Moreover, $\forall a\in X$, we also have 
\begin{equation}
\partial q_\gamma(a)=\frac{1}{\gamma} J(a),
\label{eq:scaled-subdifferential}
\end{equation}
and
\begin{equation}
J(a)
=\left\{
b^*\in X^*\;\middle|\;
\langle a,b^*\rangle=\|a\|^2=\|b^*\|_*^2
\right\}.
\label{eq:duality-mapping-explicit}
\end{equation}
\end{lemma}

\begin{proof}
Fix $b^*\in X^*$, and from the duality inequality and the scalar Young inequality, we can obtain that 
\begin{align*}
\langle b,b^*\rangle-\frac{\|b\|^2}{2\gamma}
&\leq
\|b\|\|b^*\|_*-\frac{\|b\|^2}{2\gamma}
\\
&\leq
\frac{\gamma}{2}\|b^*\|_*^2
\qquad \forall b\in X,
\end{align*}
and $q_\gamma^*(b^*)\leq\frac{\gamma}{2}\|b^*\|_*^2$. If $b^*\neq0$, then, $\forall \delta>0$, there exists a unit vector $v_\delta\in X$ such that
\begin{equation*}
\langle v_\delta,b^*\rangle>\|b^*\|_*-\delta.
\end{equation*}
Evaluating the supremum in Eq. (\ref{eq:conjugate-definition}) at $b=\gamma\|b^*\|_*v_\delta$, thus we have
\begin{equation*}
q_\gamma^*(b^*)
>\frac{\gamma}{2}\|b^*\|_*^2
-\gamma\delta\|b^*\|_*.
\end{equation*}
Let $\delta\to 0$, we can prove the reverse inequality. The case $b^*=0$ is obvious. Thus, Eq. (\ref{eq:quadratic-conjugate}) holds without requiring norm attainment.  

Next, let $a\in X$. By the subgradient inequality, we infer
\begin{align*}
b^*\in\partial q_\gamma(a)
&\Longleftrightarrow
q_\gamma(b)\geq q_\gamma(a)+\langle b-a,b^*\rangle
\quad \forall b\in X\\
&\Longleftrightarrow
\gamma b^*\in J(a),
\end{align*}
which proves Eq. (\ref{eq:scaled-subdifferential}). Finally, from Fenchel-Young equality, we know that 
\begin{equation*}
b^*\in J(a)
\Longleftrightarrow
\frac{1}{2}\|a\|^2+\frac{1}{2}\|b^*\|_*^2
=\langle a,b^*\rangle.
\end{equation*}
Equality holds precisely when $\|a\|=\|b^*\|_*$ and $\langle a,b^*\rangle=\|a\|\|b^*\|_*$, which is equivalent to Eq. (\ref{eq:duality-mapping-explicit}).
\end{proof}

Thus, based on the above results, we can now characterize the graph points at which the residual vanishes, as follows. 

\begin{corollary}[Vanishing residual]
\label{cor:certificate}
Let $A:X\rightrightarrows X^*$ be an operator. Fix $(x,u^*)\in X\times X^*$, a positive number $\gamma$, and a graph point $(y,y^*)\in\gra A$. The following statements are equivalent:
\begin{align}
&\mathcal R_\gamma(y-x,y^*-u^*)=0;
\label{eq:certificate-residual}\\
&\gamma(u^*-y^*)\in J(y-x);
\label{eq:certificate-duality}\\
&h_A(x,u^*;y,y^*)
=\frac{1}{2}D_\gamma^2((x,u^*),(y,y^*)).
\label{eq:certificate-local}
\end{align}
\end{corollary}

\begin{proof}
Set $a=y-x$ and $b^*=\gamma(u^*-y^*)$, then
\begin{align*}
\gamma\mathcal R_\gamma(y-x,y^*-u^*)
&=\frac{1}{2}\|a\|^2+\frac{1}{2}\|b^*\|_*^2-\langle a,b^*\rangle.
\end{align*}

By Eq.~(\ref{eq:quadratic-conjugate}) with $\gamma=1$ and the Fenchel-Young equality, we infer 
\begin{align*}
\frac{1}{2}\|a\|^2
+\frac{1}{2}\|b^*\|_*^2
-\langle a,b^*\rangle
=0
&\Longleftrightarrow
b^*\in\partial q(a)\\
&\Longleftrightarrow
b^*\in J(a).
\end{align*}
Therefore, we know that $\mathcal R_\gamma(y-x,y^*-u^*)=0\Longleftrightarrow\gamma(u^*-y^*)\in J(y-x)$, which proves the equivalence of Eq.~(\ref{eq:certificate-residual}) and Eq.~ (\ref{eq:certificate-duality}).

Moreover, from Proposition~\ref{prop:decomposition}, we have 
\begin{center}
$h_A(x,u^*;y,y^*)
+\mathcal R_\gamma(y-x,y^*-u^*)
=
\frac{1}{2}
D_\gamma^2((x,u^*),(y,y^*)).$
\end{center}
Consequently, we have 
\begin{center}
$
\mathcal R_\gamma(y-x,y^*-u^*)=0
\Longleftrightarrow
h_A(x,u^*;y,y^*)
=
\frac{1}{2}
D_\gamma^2((x,u^*),(y,y^*)),
$
\end{center}
which proves the equivalence with Eq.~(\ref{eq:certificate-local}).
\end{proof}

Corollary~\ref{cor:certificate} distinguishes exact residual vanishing from the approximate residual condition used in Theorem~\ref{thm:main}. The approximate condition only requires the residual infimum to be zero, whereas exact residual vanishing requires this infimum to be attained at a graph point.

\subsection{Exact attainment and residual approximation}

Next, we study whether the vanishing of the residual infimum implies its attainment. To this end, we consider the set of graph points at which the residual vanishes as follows.


\begin{definition}[Zero set of the residual]
\label{def:certificate-set}
Let $A:X\rightrightarrows X^*$ be an operator with nonempty graph. For every $(x,u^*)\in X\times X^*$ and every $\gamma>0$, we define
\begin{equation}
\mathcal C_{A,\gamma}(x,u^*)
=
\left\{
(y,y^*)\in\gra A
\;\middle|\;
\mathcal R_\gamma(y-x,y^*-u^*)=0
\right\}.
\label{eq:certificate-set}
\end{equation}
\end{definition}

Since $y^*\in Ay$, the condition in Eq. (\ref{eq:certificate-duality}) can be written as
\begin{equation}
0\in\gamma(A-u^*)y+J(y-x),
\label{eq:metricresolvent-inclusion}
\end{equation}
Eq.~(\ref{eq:metricresolvent-inclusion}) is the metric resolvent inclusion used in the proof of the result established in reflexive Banach spaces \cite{combettesmayrand2026}. Solving this inclusion is equivalent to finding a graph point at which the residual vanishes, but the proof of Theorem~\ref{thm:main} only requires graph points with arbitrarily small residual.  This motivates the following enlarged duality mapping, which describes approximate residual vanishing.  

\begin{definition}[Enlarged duality mapping]
\label{def:enlarged-duality}
For $\eta\geq0$ and $a\in X$, define
\begin{equation}
J_\eta(a)
=\left\{
b^*\in X^*\;\middle|\;
\frac{1}{2}\|a\|^2+\frac{1}{2}\|b^*\|_*^2
\leq\langle a,b^*\rangle+\eta
\right\}.
\label{eq:enlarged-duality}
\end{equation}
This is the enlarged duality mapping used in perturbation theory for maximally monotone operators on nonreflexive spaces \cite{marquesalvessvaiter2011}.
\end{definition}

With this enlargement, a small residual is equivalent to the following duality-mapping condition.

\begin{lemma}[Small residual]
\label{lem:approximate-certificate}
Let $(x,u^*)\in X\times X^*$, let $\gamma>0$, and fix $(y,y^*)\in X\times X^*$. Then, $\forall \varepsilon\geq0$,
\begin{equation}
\mathcal R_\gamma(y-x,y^*-u^*)\leq\varepsilon
\quad\Longleftrightarrow\quad
\gamma(u^*-y^*)\in J_{\gamma\varepsilon}(y-x).
\label{eq:approximate-certificate}
\end{equation}
\end{lemma}

\begin{proof}
Set $a=y-x$ and $b^*=\gamma(u^*-y^*)$. Then, we infer 
\begin{equation*}
\gamma\mathcal R_\gamma(y-x,y^*-u^*)
=\frac{1}{2}\|a\|^2
+\frac{1}{2}\|b^*\|_*^2
-\langle a,b^*\rangle.
\end{equation*}
By Eq. (\ref{eq:enlarged-duality}), the displayed quantity is at most $\gamma\varepsilon$ if and only if $b^*\in J_{\gamma\varepsilon}(a)$, which is Eq. (\ref{eq:approximate-certificate}).
\end{proof}

Thus, by utilizing the equivalence in Lemma~\ref{lem:approximate-certificate}, we can distinguish approximate residual vanishing from exact residual attainment. Moreover,  we also provide the following example showing that a vanishing residual infimum does not necessarily imply exact residual attainment, even for the zero operator on $c_0$.

\begin{example}[Nonattainment on $c_0$]
\label{ex:nonattainment}
Let $X=c_0$ be the Banach space of real sequences converging to zero, identify $X^*$ with $\ell^1$, and take $A=0=\partial 0$. The operator $A$ is maximally monotone and quasidense because it is the subdifferential of a proper lower semicontinuous convex function \cite{simons2016}. Thus, by the equivalence between quasidensity and type~(NI), $A$ is of type~(NI) \cite{simons2020}. We set $x=0$, $\gamma=1$, and
\begin{equation}
u^*=(2^{-1},2^{-2},2^{-3},\ldots)\in\ell^1.
\label{eq:c0-dual-point}
\end{equation}
Then $\|u^*\|_1=1$. If the residual vanished at a graph point, Corollary~\ref{cor:certificate} would give a point $z\in c_0$ such that $u^*\in J(z)$. By Eq. (\ref{eq:duality-mapping-explicit}), this would imply
\begin{equation}
\|z\|_\infty=1
\quad\text{and}\quad
\sum_{n=1}^{\infty}2^{-n}z_n=1.
\label{eq:c0-exact-conditions}
\end{equation}
Since $2^{-n}>0$ and $|z_n|\leq1$, Eq. (\ref{eq:c0-exact-conditions}) implies that $z_n=1$ for $\forall n\in\mathbb N$. The constant sequence $(1,1,\ldots)$ does not belong to $c_0$, thus the residual cannot vanish at a graph point.

For each positive integer $N$, we define
\begin{equation}
z^{(N)}=(\underbrace{1,\ldots,1}_{N\text{ entries}},0,0,\ldots).
\label{eq:c0-approximate-vector}
\end{equation}
Then $z^{(N)}\in c_0$, $\|z^{(N)}\|_\infty=1$, and
\begin{equation}
\langle z^{(N)},u^*\rangle=1-2^{-N}.
\label{eq:c0-pairing}
\end{equation}
Substituting Eq. (\ref{eq:c0-pairing}) gives
\begin{equation}
\mathcal R_1(z^{(N)},-u^*)
=\frac{1}{2}\|z^{(N)}\|_\infty^2
+\frac{1}{2}\|u^*\|_1^2
-\langle z^{(N)},u^*\rangle
=2^{-N}.
\label{eq:c0-residual}
\end{equation}
Thus, $\delta_{A,1}(0,u^*)=0$, but this infimum is not attained on $\gra A$. Given $ \varepsilon>0$, one can choose $N$ such that $2^{-N}\leq\varepsilon$. 
For such an $N$, Lemma~\ref{lem:approximate-certificate} and Eq.~(\ref{eq:c0-residual}) give 
\begin{equation*}
u^*\in J_{2^{-N}}(z^{(N)})\subset J_{\varepsilon}(z^{(N)}),
\end{equation*}
where $H_A(0,u^*)=+\infty$, thus this example is used only to show nonattainment.
\end{example}

The preceding example shows that exact attainment and vanishing of the residual infimum are distinct conditions. The following remark records their logical relation. 

\begin{remark}[Attainment and approximation]
\label{rem:certificate-hierarchy}
Let $A:X\rightrightarrows X^*$ be an operator with nonempty graph, let $(x,u^*)\in X\times X^*$, and let $\gamma>0$. Since the residual is nonnegative, Definitions~\ref{def:weighted-defect} and~\ref{def:certificate-set} give the first equivalence and the first implication below, while Corollary~\ref{cor:defect-form} gives the last implication as follows. 
\begin{equation}
\begin{aligned}
\mathcal C_{A,\gamma}(x,u^*)\neq\emptyset
&\quad\Longleftrightarrow\quad
\min_{(y,y^*)\in\gra A}
\mathcal R_\gamma(y-x,y^*-u^*)=0\\
&\quad\Longrightarrow\quad
\delta_{A,\gamma}(x,u^*)=0\\
&\quad\Longrightarrow\quad
H_A(x,u^*)
\geq\frac{1}{2}d_{\gra A,\gamma}^2(x,u^*).
\end{aligned}
\label{eq:certificate-hierarchy}
\end{equation}
In Example~\ref{ex:nonattainment}, $\delta_{A,1}(0,u^*)=0$ and $\mathcal C_{A,1}(0,u^*)=\emptyset$. Thus, a zero residual infimum indeed does not imply the existence of a graph point at which the residual vanishes.

Furthermore, in the reflexive setting of \cite{combettesmayrand2026}, metric resolvent theory yields a graph point at which the residual vanishes, whereas for maximally monotone operators of type~(NI) on an arbitrary Banach space, Corollary~\ref{cor:defect-form} gives only the vanishing of the residual infimum, and Theorem~\ref{thm:main} shows that this is sufficient for the sharp lower bound. These results show that the sharp estimate itself does not require reflexivity. 
\end{remark}

\subsection{A Fenchel-Young consequence}

The main theorem also gives a direct consequence for convex functions. For $\varphi\in\Gamma_0(X)$, define the Fenchel-Young gap by
\begin{equation}
L_\varphi(x,u^*)
=
\varphi(x)+\varphi^*(u^*)-\langle x,u^*\rangle.
\label{eq:fenchel-young-gap-function}
\end{equation}
The standard comparison
\begin{equation}
L_\varphi(x,u^*)
\geq
H_{\partial\varphi}(x,u^*)
\label{eq:fenchel-young-haraux}
\end{equation}
is recorded in \cite{combettesmayrand2026}. 

Combining Eq. (\ref{eq:fenchel-young-haraux}) with Theorem~\ref{thm:main} gives the graph-distance estimate for the Fenchel-Young gap as follows.  

\begin{corollary}[A bound for the Fenchel-Young gap in terms of distance to the graph]
\label{cor:fenchel-young-distance}
Let $\varphi\in\Gamma_0(X)$. Then, for every $(x,u^*)\in X\times X^*$ and every $\gamma>0$,
\begin{equation}
\varphi(x)+\varphi^*(u^*)-\langle x,u^*\rangle
\geq
\frac{1}{2}d_{\gra\partial\varphi,\gamma}^2(x,u^*).
\label{eq:fenchel-young-distance}
\end{equation}
\end{corollary}

\begin{proof}
The subdifferential $\partial\varphi$ is maximally monotone and quasidense \cite{simons2016}. Thus, by the equivalence between quasidensity and type~(NI), it is of type~(NI) \cite{simons2020}. Combining Eq.~(\ref{eq:fenchel-young-haraux}) with Theorem~\ref{thm:main}, we obtain 
\begin{equation*}
L_\varphi(x,u^*)
\geq H_{\partial\varphi}(x,u^*)
\geq\frac{1}{2}d_{\gra\partial\varphi,\gamma}^2(x,u^*),
\end{equation*}
which proves Eq. (\ref{eq:fenchel-young-distance}). 
\end{proof}

\section{Conclusion}
\label{sec:conclusion}

In this paper, we proved that the sharp lower bound with constant $\frac{1}{2}$ holds for every maximally monotone operator of type~(NI) on an arbitrary real Banach space, and that $\frac{1}{2}$ remains the best uniform constant. The exact pointwise decomposition yields the sharp bound from the vanishing of the residual infimum, without requiring a graph point at which the residual vanishes. The decomposition also yields a lower bound involving the residual infimum for every operator with a nonempty graph, and for maximally monotone operators it yields a new characterization of type~(NI) in terms of the Haraux function.
The example on $c_0$ shows that exact residual attainment is strictly stronger than the vanishing of the residual infimum required for the sharp bound.





\bibliographystyle{IEEEtran}
\bibliography{refer}

\end{document}